\documentclass{amsart}

\usepackage{amsthm,amsmath,amssymb}
\usepackage{graphicx,color}
\usepackage{enumitem, theoremref}
\usepackage{comment}

\usepackage{theoremref}
\newtheorem{theorem}{Theorem}
\newtheorem{lemma}[theorem]{Lemma}

\newtheorem{corollary}[theorem]{Corollary}

\newtheorem{proposition}[theorem]{Proposition}

\newcommand{\Z}{\mathbb{Z}}

\newcommand{\R}{\mathbb{R}}

\renewcommand{\P}{\mathbb{P}}
\newcommand{\E}{\mathbb{E}}
\newcommand{\Po}{\mathbf{P}}
\newcommand{\Eo}{\mathbf{E}}
\newcommand{\Pu}{\mathrm{P}}
\newcommand{\Eu}{\mathrm{E}}

\def\beq{\begin{equation}}
\def\eeq{\end{equation}}
\def\bay{\begin{array}}
\def\eay{\end{array}}
\def\eps{\epsilon}

\newcommand{\1}{\mathbf{1}}

\newcommand{\G}{\mathcal{G}}
\newcommand{\cir}{\bullet}
\newcommand{\tom}[1]{#1}

\author[E. Foxall]{Eric Foxall}
\address{Department of Mathematical and Statistical Sciences, Arizona State University}
\email{e.t.foxall@gmail.com}

\author[T. Hutchcroft]{Tom Hutchcroft}
\address{Department of Pure Mathematics and Mathematical Statistics, University of Cambridge} 
\email{t.hutchcroft@maths.cam.ac.uk}

\author[M. Junge]{Matthew Junge}
\address{Department of Mathematics, Duke University}
\email{jungem@math.duke.edu}

\title{Coalescing random walk on unimodular graphs}
\begin{document}
\maketitle

\begin{abstract}
Coalescing random walk on a unimodular random rooted graph for which the root has finite expected degree visits each site infinitely often almost surely. A corollary is that an opinion in the voter model on such graphs has infinite expected lifetime. Additionally, we deduce  an adaptation of our main theorem that holds uniformly for coalescing random walk on finite random unimodular graphs with degree distribution stochastically dominated by a probability measure with finite mean. 
\end{abstract}

\section{Introduction}

\emph{Coalescing random walk} (CRW) starts with one particle at each vertex of a locally finite, connected and undirected graph.
 Each particle then performs a continuous time edge-driven random walk, jumping along each edge adjacent to its present location according to a unit intensity Poisson process. Call the process \emph{site recurrent} if every site is visited infinitely often almost surely.
 
   Griffeath proved that CRW is site recurrent on $\mathbb Z^d$ \cite{griffeath}. Benjamini et.\ al.\ recently showed CRW is site recurrent on any bounded degree graph \cite{crw}. As for unbounded degree graphs, \cite[Theorem 2 (ii)]{crw} shows the process is site recurrent on any Galton-Watson tree whose offspring distribution has an exponential tail.

Our goal is to show that CRW is site recurrent on any \emph{unimodular random rooted graph} for which the root has finite expected degree. This is a  general class of random graphs that arises frequently in applications. A quick corollary is that opinions in the voter model on such graphs have infinite expected lifetime.  \tom{Previous works relating to CRW and the voter model on random graphs include \cite{MR3052399,tom}. In particular, \cite{tom} establishes that every pair of particles eventually coalesce a.s.\ in any \emph{recurrent} unimodular random rooted graph.
}
   
Before we give the definition, we start with a few interesting examples of transient unimodular random graphs with unbounded, but finite expected degree.
   
\begin{enumerate}[label = (\roman*)]
	\item An augmented Galton-Watson tree whose offspring distribution has finite expected mean. Augmented means that an extra child is added to the root. Note this is much more general than the exponential tail requirement needed in \cite[Theorem 2 (ii)]{crw}.
	\tom{
	\item Various random graphs obtained from $\Z^d$ or $\R^d$ for $d\geq 3$, including: Random geometric graphs defined in terms of point processes in $\R^d$ \cite{MR1986198,MR3406589}; Supercritical long-range percolation clusters in $\Z^d$; Graphs obtained from $\Z^d$ by replacing each edge with a random number of parallel edges with finite mean in some translation-invariant way. 
	\item Curien's planar stochastic hyperbolic triangulations, which are transient versions of the uniform infinite planar triangulation \cite{MR3520011}; Various random graphs obtained from hyperbolic space $\mathbb{H}^d$ for $d\geq 2$, such as hyperbolic random geometric graphs built from point processes in $\mathbb{H}^d$ \cite{benjamini2014anchored}.
	}
\end{enumerate}
\tom{
Furthermore, combining our result with a simple compactness argument also yields that, whenever $G$ is a \emph{finite} graph whose degree distribution is stochastically dominated by some integrable reference distribution $\mu$, CRW ``looks recurrent" on $G$ from the perspective of most vertices, in a quantitative way that depends only on the reference distribution $\mu$. See Corollary \ref{cor:quantitative} for a precise statement.
}

\subsection*{Definitions and theorem statement}
To state our theorem we develop the framework and definitions more carefully. As in \cite{ergodic}, a \emph{rooted graph} is a pair $(G,\rho)$ with $G=(V,E)$ and $\rho \in V$. An isomorphism between two rooted graphs is a graph isomorphism that maps the roots of the graphs to each other. Let $\G_{\cir}$ denote the set of isomorphism classes of locally finite, connected rooted graphs with no loops or multiple edges. Similarly we have $\G_{\cir\cir}$, the set of isomorphism classes of bi-rooted graphs $(G,x,y)$. When there is no cause for confusion, we use $(G,\rho)$ to refer interchangeably to an element of $\G_{\cir}$, or to a representative of its class, and similarly for $\G_{\cir\cir}$. 

Coalescing random walk is well-defined on isomorphism classes of rooted graphs, since its distribution does not depend on the choice of representative. The same is true for random walk, with the understanding that the starting point of the walk may depend on the isomorphism class, and is mapped to an isomorphic location when passing between representatives. These two technicalities can be surmounted by choosing a root as the starting point of the walk.

A random rooted graph $(G,\rho)$ is a random variable with values in $\G_{\cir}$. Following \cite{ergodic} we use $\Po$ and $\Eo$ to denote probability and expectation for a random rooted graph. We use $\Pu_{\rho}^G$ and $\Eu_{\rho}^G$ to denote the \emph{quenched} probability and expectation of a process taking place on $(G,\rho)$, \tom{that is, the conditional law of the process given $(G,\rho)$}. We use $\P$ and $\E$ to denote \tom{probabilities and expectations with respect to \emph{annealed} measure, that is, the} joint distribution of $(G,\rho)$ and the process on $(G,\rho)$.

Before stating our theorem we give two definitions from \cite{ergodic}.
Let $(G,\rho)$ denote a random rooted graph and let $(X_n)_{n\geq 0}$ denote random walk on $(G,\rho)$ with $X_0=\rho$. Then $(G,\rho)$ is \emph{stationary} if 
$$(G,X_0) = (G,X_1)\quad\textrm{in distribution}.$$
To introduce unimodularity we require something called the \emph{mass-transport principle} or MTP for short. A random rooted graph $(G,\rho)$ satisfies the MTP if for every Borel positive function $F\colon \G_{\cir\cir} \to \R_+$ we have
\begin{equation}\label{eq:mtp}
\Eo \left[ \sum_{x \in V}F(G,\rho,x) \right] = \Eo \left[ \sum_{x \in V}F(G,x,\rho) \right].
\end{equation}
A random rooted graph $(G,\rho)$ is \emph{unimodular} if it satisfies \eqref{eq:mtp}.
 We think of $F(G,\rho,x)$ as an amount of mass sent from $\rho$ to $x$.  The above says that the average mass sent from $\rho$ to other vertices is the same as the average mass received by $\rho$.
Note that every finite graph can be made into a unimodular random rooted graph by rooting it at a uniformly random vertex.
Unimodular random rooted graphs were introduced by Benjamini and Schramm \cite{MR1873300} and developed systematically by Aldous and Lyons \cite{aldous_lyons}.

With this notation we can state our theorem.

\begin{theorem}\thlabel{thm:main}
CRW is site recurrent $\mathbb P$-almost surely on any \tom{unimodular random rooted} graph $(G,\rho)$ for which $\mathbf E[\deg(\rho)]<\infty$. 
\end{theorem}

It is quite easy to construct counterexamples to show that the assumption $\mathbf E[\deg(\rho)]<\infty$ is necessary, for example by taking a canopy tree and replacing each edge at height $n$ with a large number of parallel edges.

The proof of \thref{thm:main} uses the well-known duality between CRW and the voter model. The \emph{voter model} is a process on $(G,\rho)$ in which each site has a unique opinion which it spreads to neighbors at rate 1 (along each edge). Let $\zeta_t^\rho$ be the set of vertices at time $t$ with the opinion started at $\rho$. The voter model is dual to CRW in the sense that $\zeta_t^\rho$ is equal in distribution to the set of particles in CRW that have coalesced and occupy site $\rho$ at time $t$. We obtain \thref{thm:main} by proving a linear bound on the annealed second moment of $|\zeta_t^\rho|$.

\begin{proposition}\thlabel{prop:main} 
Let $(G,\rho)$ be a unimodular random rooted graph. Then $\mathbb E |\zeta_t^\rho|^2 \leq 1 + 2t\mathbf E[\deg(\rho)]$ for every $t\geq 0$.
\end{proposition}

To prove \thref{prop:main} we start by defining the voter model and some important duality relationships in Section \ref{sec:voter}. We then use the MTP in Lemma \ref{lem:sizebias} to show that the size of the voter model cluster \emph{currently occupying} $\rho$, denoted $|\zeta^{(\rho)}|$, has the size-biased distribution of  $|\zeta_t^{\rho}|$.
A different duality relation at \eqref{eq:duality2} relates the size of the voter model cluster occupying $\rho$ to the total number of CRW particles coalesced with the one started at $\rho$. 
We use this in \thref{lem:clusbnd} to show that the expected number of other particles coalesced to a given particle is bounded by two times the integral of the vertex degrees along the particle's random walk path. This is accomplished with a coupling that gives priority to the path followed by the particle started from $\rho$ and ignores certain collisions. Using the stationarity and reversibility observed in Lemma \ref{lem:unistat}, we obtain the bound $\E | \zeta^{(p)}_t| \leq \tom{1+2t\mathbf{E}[\deg(\rho)]}$ for all $t \geq 0$. The size-biasing relationship observed in Lemma \ref{lem:sizebias} implies that $\E |\zeta^{(p)}|= \E[ \, | \zeta^{\rho}_t|^2 \, ]$, which gives the desired second moment bound.

\thref{thm:main} follows in an elementary way from duality and \thref{prop:main}, so we give the proof now.

\begin{proof}[Proof of \thref{thm:main}]
To compress our notation we let $E$ denote $E_\rho^G$, $P$ denote $P_G^\rho$, and $\zeta_t$ denote $\zeta_t^\rho$. 
 Let $c=1+2\mathbf{E}[\deg(\rho)]$. Fix $m> c$ and let $\epsilon = c/m$. Applying  Markov's inequality to the random variable $E |\zeta_t|^2 = \mathbb E [ \,|\zeta_t|^2 \mid (G,\rho)]$, we obtain that
	$$\P( E | \zeta_t|^2 \leq mt ) \geq 1 - \epsilon >0$$
for every $t\geq 1$. For integer $k\ge 1$ let $A_k = \{E|\zeta_k|^2 \leq mk\}$ so that $P(A_k) \geq 1- \epsilon$ for all $k$. By the reverse Fatou's lemma applied to $\1(A_k)$, $\limsup_k \P(A_k) \leq \P( \limsup_k A_k )$, and thus 
$$1 - \epsilon \leq \P(A_k \text{ occurs infinitely often}).$$ 
So, with probability at least $1- \epsilon$, $ E |\zeta_{t_i}|^2 \leq m t_i$ for some sequence of integers $t_i \uparrow \infty$. By \thref{lem:mtg}, we have that $E |\zeta_t| =1$. The Cauchy-Schwarz inequality then yields
	\begin{align}
	P( |\zeta_{t_i}| > 0 ) \geq \frac{(E |\zeta_{t_i}|)^2}{E|\zeta_{t_i}|^2} \geq \frac 1{mt_i}\label{eq:2nd}
	\end{align}
	 for each $t_i$. \tom{Let $i_0=0$ and for each $k\geq 1$ let $i_k$ be minimal such that $t_{i_k} \geq 2 t_{i_{k-1}}$. Let $s_k=t_{i_{k}}$.} Since $0$ is an absorbing state, $t\mapsto P(|\zeta_{t}| > 0)$ is non-increasing \tom{and we deduce that}
% Using this, the above estimate, and Lemma \ref{lem:integral} we find that
\begin{align}\int_0^\infty  P(|\zeta_{t}| > 0 ) dt \geq \ \frac 1 m \sum_{k=1}^\infty \frac{s_{k+1} - s_k}{s_{k+1}} \geq \frac{1}{m} \sum_{k=1}^\infty \frac{1}{2} =  \infty \label{eq:voter_survival}
	\end{align}
with probability at least $1-\epsilon$. Let $B_t$ be the event that in CRW there is a particle at $\rho$ at time $t$. By the duality relationship described in \eqref{eq:duality}, $P(B_t) = P(|\zeta_t|>0)$ and so
	$$\P\left(\int_0^{\infty} P(B_t)dt = \infty\right) \ge 1-\eps.$$
Since $\eps$ can be made small by choosing $m$ large, the above has $\P$-probability 1. Moreover, on a fixed graph $G$, it is known (see for example \cite{crw}) that $\int_0^{\infty}P(B_t)dt=\infty$ implies $\rho$ is recurrent (i.e., is occupied infinitely often as $t\to\infty$) almost surely. Combining these observations, it follows that $\rho$ is recurrent $\P$-a.s. Since any property that holds a.s.\ for the root of a unimodular graph holds for all vertices a.s. \tom{\cite[Lemma 2.3]{aldous_lyons}}, it follows that all vertices are recurrent $\P$-a.s. 
\end{proof}

We now give two corollaries of our theorem. 
% First off, there is some interest in the voter model on unimodular graphs \cite{tom}.
 The expression at \eqref{eq:voter_survival} is the expected survival time of the opinion started at $\rho$ in the voter model on $G$. From it we obtain the following corollary.

\begin{corollary}
\tom{Let $(G,\rho)$ be a unimodular random rooted graph with $\mathbf{E} [\deg(\rho)]<\infty$.}
	Then the \tom{quenched} expected survival time of the opinion started at $\rho$ is infinite \tom{a.s.}
\end{corollary}

\medskip

\tom{
As remarked previously, we can use compactness arguments to derive ``uniform'' versions of Theorem \ref{thm:main} that are already interesting in the context of finite graphs. 
We denote by $\sigma_t(v)$ the first time after $t$ that $v$ is visited by a particle in CRW.

\begin{corollary}
\label{cor:quantitative}
Let $\mu$ be a probability measure on $\mathbb N$ with finite mean. Then there exists a function $f_\mu(t,u):[0,\infty)\times[0,\infty)\to [0,1]$ with
\[ \lim_{u\to\infty} f_\mu(t,u) = 0 \qquad \forall t\geq 0\]
 such that 
 whenever $(G,\rho)$ is a unimodular random rooted graph such that the law of $\deg(\rho)$ is stochastically dominated by $\mu$, we have that
\[
\P\bigl( \sigma_t(\rho) \geq t+u\bigr) \leq f_\mu(t,u)
\]
for every $t,u \geq 0$.
\end{corollary}

\begin{proof}
Let $M$ be the set of laws of unimodular random rooted graphs with root degree stochastically dominated by $\mu$. Then $M$ is compact with respect to the weak topology \cite[Section 3.2.1]{CurienNotes}, and since $\mathcal{G}_\bullet$ is a Polish space when equipped with the local topology \cite[Theorem 2]{CurienNotes}, $M$ is sequentially compact also. 

Let $\sigma_t^{(r)}(\rho)$ be the first time after $t$ that $\rho$ is visited by a particle in the modified coalescing random walk in which particles are killed upon leaving the ball of radius $r$ around $\rho$. Then for each $t,u\geq 0$, the function 
\[\nu\mapsto \nu\left[ P^G_\rho\left(\sigma_t^{(r)}(\rho) \geq t+u\right)\right]\] is clearly continuous with respect to the weak topology on measures on $\mathcal{G}_\bullet$. Since $P^G_\rho(\sigma_t(\rho) \geq t+  u) = \inf_{r\geq 1} P^G_\rho(\sigma_t^{(r)}(\rho) \geq t+u)$ it follows that the function 
\[\nu\mapsto \nu\left[ P^G_\rho\left(\sigma_t(\rho) \geq t+u\right)\right]\] 
is an infimum of continuous functions and is therefore upper semi-continuous with respect to the weak topology on the space of probability measures on $\mathcal{G}_\bullet$. Thus, it obtains its maximum on $M$, and we denote this maximum by $f_\mu(t,u)$. Clearly $f_\mu(t,u)$ is decreasing in $u$. 

Suppose for contradiction that $\inf_{u\geq0} f_\mu(t,u)\geq \eps$ for some $\eps > 0$ and $t\in [0,\infty)$. Then for each $n\geq 1$ there exists $\nu_n\in M$ such that 
\[
\nu_n\left[ P_\rho^G\left(\sigma_t(\rho) \geq t+m \right)\right]  \geq \eps \text{ for all $1\leq m \leq n$.}
\]
By compactness of $M$, the measures $\nu_n$ have a subsequential limit $\nu \in M$, and by upper semi-continuity we have that
\[
\nu\left[ P_\rho^G\left(\sigma_t(\rho) \geq t+m \right)\right]  \geq \eps \text{ for all $m\geq 1$.}
\]
This contradicts Theorem \ref{thm:main}.
\end{proof}
}

\section{Establishing \thref{prop:main}}

We follow the outline described in the introduction just proceeding the statement of \thref{prop:main}.

\subsection{Voter model duality} \label{sec:voter}

We describe the graphical representation that allows us to construct and analyze coalescing random walk as well as a related \emph{voter model}. 
We do this on a fixed graph; to obtain a reasonable joint measure $\P$ of the random graph and CRW, it suffices to define the elements of the graphical representation (i.e., the processes $U_{(v,w)}$ given below) recursively over finite rooted graphs.\\

Let $G=(V,E)$ be an undirected, locally finite graph on which the continuous time edge-driven walk is non-explosive, i.e.\ makes at most finitely many jumps in any finite time interval a.s. This property always holds a.s.\ on unimodular random rooted graphs with finite expected degree \cite[Corollary 4.4]{aldous_lyons}, and can therefore be safely assumed throughout our analysis.
% For a fixed undirected, locally finite graph $(V,E)$, 
Double the edge set \tom{of $G$} to form the \tom{set of} directed edges $F = \{(v,w)\colon v,w \in E\}$. Define a family $\{U_{(v,w)}\colon(v,w) \in F\}$ of independent unit intensity Poisson point processes. Then a particle dropped at a spacetime location $(v,s)$ follows a unique path $\Theta_t(v,s)\colon[s,\infty) \to V$ defined by updating to $\Theta_t(v,s) = u$, any time that $\Theta_{t^-}(v,s)=w$ and $t \in U_{(w,u)}$; local finiteness \tom{and non-explosivity} ensures that the jump times are \tom{a discrete subset of $[0,\infty)$, and hence that $\Theta_t(v,s)$ is well-defined for every $v\in V$, $s \in [0,\infty)$ and $t\geq s$ a.s.} Coalescing random walk, labelled by $\xi_t^v$ for the location of the particle that began at $v$, is defined by
$$\xi_t^v = \Theta_t(v,0).$$
We let $\xi_t = \{ w \in V \colon \exists v \text{ such that } \xi_t^v = w\}$ be the collection of all occupied sites at time $t$.

The related voter model $\zeta_t$, initialized at any time $T>0$, is defined on the time interval $[0,T]$ by letting
$$\zeta_t^v = \{w:\Theta_T(w,T-t)=v\}.$$
In particular, $\zeta_0^v = \{v\}$, and for $t>0$, $\zeta_t^v$ and $\zeta_t^w$ are disjoint if $w \ne v$. The non-explosivity assumption ensures that $V = \bigcup_v \zeta_t^v$ for $t>0$. Therefore $\{\zeta_t^v:v \in V\}$ is a partition of $V$, so we can think of the sets $\zeta_t^v$ as clusters vying for control of the territory $V$, with $\zeta_t^v$ being the cluster that started as $\{v\}$. Labelling clusters by their starting vertex, we could also define the voter model by $\zeta_t(v) = \Theta_T(v,T-t)$, with $\zeta_t(v)$ equal to the label of the cluster that contains $v$ at time $t$.

From the graphical representation, looking backwards in time we see \tom{that} the voter model is Markov with the following transition rule:
for each $(v,w) \in F$, at rate $1$, $\zeta_t(v) = \zeta_{t^-}(w)$, or equivalently, the cluster containing $w$ swallows $v$. There are various connections between $\xi_t$ and $\zeta_t$. The most basic is that
\begin{align}
\zeta_T^v = \{w\colon\xi_T^w = v\}. \label{eq:duality} 
\end{align}
If we define $\smash{\zeta_t^{(v)} = \zeta_t^{\zeta_t(v)}}$ to be the cluster containing $v$, we note also that
\begin{align}
\smash{\zeta_T^{(v)} = \{w\colon\xi_T^w = \xi_T^v\}},\label{eq:duality2} 	
\end{align}
and that $\smash{v \in \zeta_T^{\xi_T^v}}$. If we think of the voter model as being constructed separately from CRW, these equalities are in distribution and hold for any fixed value of $T\ge 0$.

%\subsection{Establishing \thref{prop:main}}
%The basic strategy is the same as in \cite{crw, griffeath}. Namely, we wish to show the occupation time of a certain vertex, in this case the root, has infinite expectation.
%More precisely, let $p_t$ be the $\P$-probability that there is a particle at $\rho$ at time $t$. We can relate this to the total occupation time of $\rho$, $\tau$, from \thref{prop:main}.

%\begin{proposition}\thlabel{prop:pt}
%	$\mathbb E [\tau] = \int_0^\infty p_t dt.$
%\end{proposition}
%
%\begin{proof}
%Recall that $\xi_t$ is the set of all occupied sites in CRW. This lets us write $\tau = \int_0^\infty \mathbf 1\{\rho \in \xi_t\} dt.$ Take expectation and apply Fubini's theorem to obtain the claimed formula.
%\end{proof}

\subsection{Proof of \thref{prop:main}}

We start by using the mass-transport principle to deduce a size-biasing property for the voter model.

\begin{lemma}\label{lem:sizebias}
Let $(G,\rho)$ be a unimodular random graph \tom{on which the continuous time edge-driven random walk is non-explosive}. Then $|\zeta_t^{(\rho)}|$ has the size-biased distribution of $|\zeta_t^{\rho}|$. That is, for \tom{every $t\geq 0$ and every integer $n\geq 0$,}
$$\P(\,|\zeta_t^{(\rho)}| = n\,) = n \P(\,|\zeta_t^{\rho}|=n\,).$$
\end{lemma}
\begin{proof}
The duality relation at \eqref{eq:duality2} ensures that for any $(G,\rho)$ we have $\rho \in \zeta_t^{\xi_t^{\rho}}$ when the voter model is initialized at time $t$. So, for $n\geq 1$ we have the disjoint union
\begin{equation}\label{eq3}
\{|\zeta_t^{(\rho)}| = n\} = \bigcup_{x \in V(G)} \{|\zeta_t^x|=n,\,\, \rho \in \zeta_t^x\}.
\end{equation}
Define $F\colon \G_{\cir\cir}\to \R_+$ by
$$F(G,\rho,x) = \Eu_G^{\rho}[\1(|\zeta_t^x|=n,\,\,\rho \in \zeta_t^x)].$$
Using the MTP as formulated at \eqref{eq:mtp} and \eqref{eq3},
\beq\label{eq4}
\P(|\zeta_t^{(\rho)}| = n) = \Eo\left[ \sum_{x \in V(G)}F(G,\rho,x) \right ] = \Eo\left[ \sum_{x \in V(G)}F(G,x,\rho) \right ] .\\
\eeq
On the event $|\zeta_t^{\rho}|=n$, $\sum_{x \in V(G)}\1(x \in \zeta_t^{\rho})=n$. Writing the indicator as a product and using Fubini's theorem,
\begin{eqnarray*}
\sum_{x \in V(G)} F(G,x,\rho) &=& \sum_{x \in V(G)} \Eu_\rho^G[\1(|\zeta_t^{\rho}|=n,\,\,x \in \zeta_t^{\rho})] \\
&=& \Eu_\rho^G\left[\sum_{x \in V(G)} \1(|\zeta_t^{\rho}|=n)\1(x \in \zeta_t^{\rho})\right] \\
&=& n\Eu_\rho^G \left[\1(|\zeta_t^{\rho}|=n)\right].
\end{eqnarray*}
Combining with \eqref{eq4},
$$\P(|\zeta_t^{(\rho)}| = n) = \Eo[n\Eu_\rho^G[(\1(|\zeta_t^{\rho}|=n)]] = n\P(|\zeta_t^{\rho}| =n),$$
as desired.
\end{proof}

On bounded degree graphs $|\zeta_t^\rho|$ is a martingale. Indeed, it transitions as a nearest-neighbor simple random walk whose jump rate is equal to 2 times the number of (undirected) boundary edges of the cluster $\zeta_t^\rho$. A concern on unbounded degree graphs is that a lack of integrability might cause the size of the cluster to instead be a strict local martingale. The next lemma shows this is not the case in our setting.

\begin{lemma}\thlabel{lem:mtg}
% Fix $(\rho,G)$. 
Let $(G,\rho)$ be a unimodular random rooted graph on which the continuous time edge-driven random walk is non-explosive. Then
 $E_\rho^G |\zeta_t^\rho| = 1$ for all $t \geq 0$ a.s. 
\end{lemma}

\begin{proof}
It follows from Lemma \ref{lem:sizebias} that $\E |\zeta_t^ \rho| = \sum_{n=0}^\infty \P( |\zeta_t^{(\rho)}| = n)  =1.$ Thus, it suffices to show that\tom{, conditional on $(G,\rho)$, the process}  $| \zeta^\rho_t |$ is a supermartingale \tom{a.s.}, since then we have that $E_\rho^G |\zeta^\rho_t | \leq 1$ a.s., and since this random variable must integrate to $1$ the claim follows. 
Every non-negative local martingale is a supermartingale by Fatou's lemma, and so it suffices to check that $|\zeta^\rho_t|$ is a local martingale. This follows easily after noting that the jump chain corresponding to the process $|\zeta_t^\rho|$ is a symmetric simple random walk absorbed at $0$.
\end{proof}

%\noindent The following result uses duality with the voter model and the fact that $|\zeta_t^{(\rho)}|$ has the size-biased distribution of $|\zeta_t^\rho|$ to connect the probability the root is occupied in CRW to the expected size of the voter model cluster occupying $\rho$.
%\begin{lemma}\label{lem:rwr}
%Let $(G,\rho)$ be a unimodular random graph. Then,
%$$p_t = \P(\zeta_t^\rho\neq \emptyset) \geq (\E[|\zeta_t^{(\rho)}|])^{-1}.$$
%\end{lemma}
%\begin{proof}
%The equality follows from the duality relation at \eqref{eq:duality}, since $x$ is occupied by a particle at time $T$ if and only if $\zeta_T^x \ne \emptyset$.
%The inequality follows from Lemma \ref{lem:sizebias}, using Jensen's inequality in the final step:
%\begin{align*}
%\P(\zeta_t^\rho \neq \emptyset) &= \sum_{n>0}\P(|\zeta_t^\rho| = n) = \sum_{n>0}n^{-1}\P(|\zeta_t^{(\rho)}|=n) = \E[|\zeta_t^{(\rho)}|^{-1}] \geq (\E[|\zeta_t^{(\rho)}|])^{-1}.
%\end{align*}
%\end{proof}
%\noindent Combining Lemma \ref{lem:rwr} with \thref{prop:pt} we have 
%\beq\label{eq:divvot}
%\mathbf E[\tau] = \int_0^\infty p_t dt \geq \int_0^{\infty}(\E[|\zeta_t^{(\rho)}|])^{-1}dt.
%\eeq

Let $(X_t)_{t \geq 0}$ denote continuous-time \emph{edge-driven} random walk, that is, $X_t$ moves along each edge at rate one. We will use the following property of unimodular random graphs.

\begin{lemma}\label{lem:unistat}
Let $(G,\rho)$ be a unimodular random \tom{rooted} graph with $\mathbf E[\deg(\rho)]<\infty$ and let $(X_t)_{t \geq 0}$ be a continuous-time edge-driven random walk with $X_0=\rho$. Then, the measure of $(G,\rho)$ is stationary and reversible for the random walk $(G,X_t)_{t \geq 0}$ on $\G_{\cir}$.
\end{lemma}

\begin{proof}
This follows from \cite[Corollaries 4.3 and 4.4]{aldous_lyons}.
\end{proof}

The next result allows us to estimate the size of $\zeta_t^{(\rho)}$ in terms of the average degree of $\xi_t^{\rho}$ over time.

\begin{proposition}\thlabel{lem:clusbnd}
Let $(G,\rho)$ be a fixed (nonrandom) element of $\G_{\cir}$ \tom{on which the continuous time edge-driven random walk is non-explosive}, and let $(X_t)_{t \geq 0}$ be a \tom{continuous time} edge-driven random walk on $(G,\rho)$ with $X_0=\rho$. Then
\begin{align}\Eu_\rho^G[ \ |\zeta_t^{(\rho)}| \ ] \leq 1 + 2\int_0^t \Eu_\rho^G[\,\deg(X_s)\,]ds.	\label{eq:prop}
\end{align}

\end{proposition}

\begin{proof}
Since the graph is fixed, use $\Pu$ and $\Eu$ to denote probability and expectation with respect to the process on $(G,\rho)$. First double up the edge set to $F = \{(x,y)\colon xy \in E\}$. Independently of $(X_t)$ let $\{U_e\colon e \in F\}$ be a family of independent Poisson point processes with unit intensity on $\R_+$, one for each directed edge. We think of the points
$$\{(e,t):t \in U_e, e \in F\}$$ as arrows on the spacetime set $G \times \R_+$.

Define a version of coalescing random walk $\xi_t$ on $G$ using $(X_t)$ and $\{U_e\}$ as follows. Recall that for $ y\in V$, $\xi_t^y$ is the position at time $t$ of the particle that began at $y$. First, set $\xi_t^x = X_t$ for $t\geq 0$. Then, for $y \neq x$, let $\xi_0^y = y$ and follow these rules to determine $\xi_t^y$ for $t>0$.
\begin{enumerate}
\item If $\xi_t^y = X_t$, then $\xi_s^y = X_s$ for $s>t$.
\item If $\xi_{t^-}^y = v \neq X_t$ and $t \in U_{(v,w)}$, then $\xi_t^y = w$.
\end{enumerate}
In other words, particles remain stuck to $(X_t)$ once they hit it, and otherwise follow the arrows given by the $\{U_e\}$. For $t\geq 0$ and $v \in V$ define the particle count $N_t$ by
$$N_t(v) = |\{w\colon \xi_t^w = v\}|.$$

First, condition on $\gamma = \{X_t\colon t\geq 0\}$. For $v \in V$ let $S_v(\gamma) = \{t\ge 0 \colon  X_t=v\}$. Then, for $t>0$ let
$$U_{\gamma}(t) = \{((v,w),s) \colon  (v,w) \in F,\, s \in U_{(v,w)} \cap S_w(\gamma) \cap [0,t]\},$$
and let $U_{\gamma} = \bigcup_{t>0}U_{\gamma}(t)$. The set $U_\gamma$ consists of all arrows pointing towards the space-time path $\gamma$. Modify $\xi_t$ so that it ignores $U_{\gamma}$, denoting the modified process by $\xi_t(\gamma)$, with particle count $N_t^{\gamma}$.
 Note that $N_t(v) \leq N_t^{\gamma}(v)$ if $v \neq X_t$, and that conditional on $\gamma$, $\xi_t(\gamma)$ is independent of $U_{\gamma}$.\\

Since particles do not escape from $X_t$, it follows that
$$N_t(X_t) = 1 + \sum_{((v,w),s) \in U_{\gamma}(t)} N_s(v) + \sum_{s \leq t\colon X_s \neq X_{s^-}}N_{s^-}(X_s).$$
Using the above inequality we deduce the upper bound
$$N_t(X_t) \leq 1 + \sum_{((v,w),s) \in U_{\gamma}(t)} N_s^{\gamma}(v) + \sum_{s \leq t\colon X_s \neq X_{s^-}}N_{s^-}^{\gamma}(X_s).$$
Almost surely, $U_{\gamma}(t)$ and $\{s \leq t\colon X_s \neq X_{s^-}\}$ are finite, and the above is a sum of finitely many terms. Conditioning, then noting that $\xi_t(\gamma)$ and thus $N_t^{\gamma}$ are conditionally independent of $U_{\gamma}$ given $\gamma$, we obtain
\begin{align}\label{eq1}
\Eu[\,N_t(X_t) \mid \gamma,\,U_{\gamma}\,] & \leq  1 + \sum_{((v,w),s) \in U_{\gamma}(t)} \Eu[\,N_s^{\gamma}(v) \mid \gamma\,] \\ \nonumber
&\qquad \qquad + \sum_{s \leq t\colon X_s \neq X_{s^-}}\Eu[\,N_{s^-}^{\gamma}(X_s)\mid \gamma\,].\nonumber
\end{align}
First we estimate $\Eu[\,N_s^{\gamma}(v) \mid \gamma\,]$ for $v \neq X_s$. Note that, conditioned on $\gamma$, $N_s^{\gamma}(v) = |\zeta_s|$ where $\{\zeta_r\colon  r \in [0,s]\}$ is a voter model with $\zeta_0 = \{v\}$. Looking backwards in time, this voter model has the transitions:
\begin{enumerate}
\item If $w \in \zeta_{r^-}$ and $r \in U_{(w,z)}$ with $z \neq X_{s-r}$ then $\zeta_r = \zeta_{r^-} \setminus \{w\}$.
\item If $w \in \zeta_{r^-}$ and $r \in U_{(z,w)}$ with $z \neq X_{s-r}$ then $\zeta_r = \zeta_{r^-} \cup \{z\}$.
\item If $w \in \zeta_{r^-}$ and $w = X_{s-r}$ then $\zeta_r = \zeta_{r^-} \setminus \{w\}$.
\end{enumerate}
Transitions 1.\ and 2.\ respectively increase and decrease $|\zeta_r|$ by $1$ at the same rate, while transition 3.\ decreases $|\zeta_r|$, so $|\zeta_r|$ is a supermartingale. Therefore,
$$\Eu[N_s^{\gamma}(v) \mid \gamma\,] = \Eu[\,|\zeta_s|\,] \leq \Eu[\,|\zeta_0|\,] = 1.$$
Next we estimate $N_{s^-}^{\gamma}(X_s)$. However, $N_{s^-}^{\gamma}(X_s) = |\zeta_s|$ for the same process, except with $\zeta_0 = \{X_0\}$. Therefore $\Eu[N_{s^-}^{\gamma}(X_s) \mid \,\gamma\,] \leq 1$ as well. Plugging into \eqref{eq1},
\begin{equation}\label{eq2}
\Eu[\,N_t(X_t) \mid \gamma,\,U_{\gamma}\,] \leq 1 + |U_{\gamma}(t)| + |\{s \leq t\colon X_s \neq X_{s^-}\}|.
\end{equation}
Conditioned on $\gamma$, $\{\,|U_{\gamma}(t)|\colon t\geq 0\}$ is a counting process with time-varying intensity $\deg(X_t)$, so
$$\Eu[\,|U_{\gamma}(t)|\,\mid \,\gamma\,] = \int_0^t \deg(X_s)ds.$$
Without conditioning on $\gamma$, $\{\,|\{s \leq t\colon X_s \neq X_{s^-}\}|\colon  t \geq 0\}$ is a counting process with adapted intensity $\deg(X_t)$.
Taking $\Eu[\,\cdot\,\mid \,\gamma\,]$ in \eqref{eq2} gives
$$\Eu[\,N_t(X_t) \mid \,\gamma\,] \leq 1 + \int_0^t \deg(X_s)ds + |\{s \leq t\colon X_s \neq X_{s^-}\}|,$$
then taking $\Eu[\,\cdot\,]$, noting the jump rate of $X_s$ is $\deg(X_s)$ and using Fubini's theorem, we find
$$\Eu[\,N_t(X_t)\,] \leq 1 + 2\int_0^t \Eu[\,\deg(X_s)\,]ds.$$
With respect to our construction, $\xi_t^{\rho} = X_t$ for $t\geq 0$. Therefore, the result follows from the duality relation between coalescing random walk and the voter model, since
\[\zeta_t^{(\rho)} \stackrel{d}{=} \{x\colon \xi_t^x = \xi_t^{\rho}\}.\qedhere\]
\end{proof}

Finally we complete the proof of  \thref{prop:main}.
\begin{proof}[Proof of  \thref{prop:main}]
% The stationarity property of
 Lemma \ref{lem:unistat} ensures that $\E[\deg(X_t)] = \mathbf E[\deg(\rho)]$ for $t\ge 0$. Apply this to  \thref{lem:clusbnd}, so when we take $\Eo[\cdot]$ of \eqref{eq:prop} and apply Fubini's theorem we obtain
$$\E[ \ |\zeta_t^{(\rho)}| \ ] \leq 1 + 2t \mathbf E[\deg(\rho)]$$
for $t \geq 0$.
 % Thus, there exists $c>0$ such that $\E[ |\zeta_t^{(\rho)} | ] \leq ct$ for all $ t\geq 1$. 
 We conclude the proof by observing that, by Lemma \ref{lem:sizebias}, we have $\E[ |\zeta_t^{\rho} |^2 ] = \E[ |\zeta_t^{(\rho)} | ].$
\end{proof}

%\section*{Acknowledgments}
%\tom{We thank Omer Angel for making us aware of Arratia's coupling of CRW and ARW.}

\bibliographystyle{alpha}
\bibliography{unimod}

\newcommand{\etalchar}[1]{$^{#1}$}
\begin{thebibliography}{BFGG{\etalchar{+}}16}

\bibitem[AL07]{aldous_lyons}
David Aldous and Russell Lyons.
\newblock Processes on unimodular random networks.
\newblock {\em Electron. J. Probab.}, 12:1454--1508, 2007.

\bibitem[BC12]{ergodic}
Itai Benjamini and Nicolas Curien.
\newblock Ergodic theory on stationary random graphs.
\newblock {\em Electron. J. Probab.}, 17:no. 93, 1--20, 2012.

\bibitem[BFGG{\etalchar{+}}16]{crw}
Itai Benjamini, Eric Foxall, Ori Gurel-Gurevich, Matthew Junge, and Harry
  Kesten.
\newblock Site recurrence for coalescing random walk.
\newblock {\em Electron. Commun. Probab.}, 21:12 pp., 2016.

\bibitem[BPP14]{benjamini2014anchored}
Itai Benjamini, Elliot Paquette, and Joshua Pfeffer.
\newblock Anchored expansion, speed, and the hyperbolic poisson voronoi
  tessellation.
\newblock {\em arXiv preprint arXiv:1409.4312}, 2014.

\bibitem[BPS12]{MR3052399}
Martin~T. Barlow, Yuval Peres, and Perla Sousi.
\newblock Collisions of random walks.
\newblock {\em Ann. Inst. Henri Poincar\'e Probab. Stat.}, 48(4):922--946,
  2012.

\bibitem[BS01]{MR1873300}
Itai Benjamini and Oded Schramm.
\newblock Recurrence of distributional limits of finite planar graphs.
\newblock {\em Electron. J. Probab.}, 6:no. 23, 13, 2001.

\bibitem[Cur]{CurienNotes}
Nicolas Curien.
\newblock Random graphs: The local convergence point of view.
\newblock Unpublished lecture notes. Available at
  https://www.math.u-psud.fr/~curien/cours/cours-RG-V3.pdf.

\bibitem[Cur16]{MR3520011}
Nicolas Curien.
\newblock Planar stochastic hyperbolic triangulations.
\newblock {\em Probab. Theory Related Fields}, 165(3-4):509--540, 2016.

\bibitem[Gri78]{griffeath}
David Griffeath.
\newblock Annihilating and coalescing random walks on $\mathbb{Z}^d$.
\newblock {\em Zeitschrift für Wahrscheinlichkeitstheorie und Verwandte
  Gebiete}, 46(1):55--65, 1978.

\bibitem[HP15]{tom}
Tom Hutchcroft and Yuval Peres.
\newblock Collisions of random walks in reversible random graphs.
\newblock {\em Electron. Commun. Probab.}, 20:6 pp., 2015.

\bibitem[Pen03]{MR1986198}
Mathew Penrose.
\newblock {\em Random geometric graphs}, volume~5 of {\em Oxford Studies in
  Probability}.
\newblock Oxford University Press, Oxford, 2003.

\bibitem[Rou15]{MR3406589}
Arnaud Rousselle.
\newblock Recurrence or transience of random walks on random graphs generated
  by point processes in {$\Bbb{R}^d$}.
\newblock {\em Stochastic Process. Appl.}, 125(12):4351--4374, 2015.

\end{thebibliography}

\end{document}